\documentclass[11pt]{amsart}
\usepackage{epsfig,amsfonts,amsthm,amssymb,latexsym,amsmath}

\newcommand{\mymod}[3]{#1 \equiv #2 \kern -0.5em \pmod{#3}}
\newcommand{\mynotmod}[3]{#1 \not \equiv #2 \kern -0.6em \pmod{#3}}

\theoremstyle{plain}
\newtheorem{theorem}{Theorem}[section]

\newtheorem{proposition}[theorem]{Proposition}

\theoremstyle{remark}

\theoremstyle{definition}

\title{A note on Modified Third-order Jacobsthal numbers}
\vspace{5pt}

\author{\scriptsize Gamaliel Cerda-Morales}
\date{}

\begin{document}
\maketitle

\vspace{-20pt}
\begin{center}
{\footnotesize Departamento de Matem\'atica, Facultad de Ciencias F\'isicas y Matem\'aticas, Universidad de Concepci\'on, Casilla 160-C, Concepci\'on, Chile. \\
E-mails: gcerda@udec.cl 
}\end{center}

\vspace{5pt}

\begin{abstract}
Modified third-order Jacobsthal sequence is defined in this study. Some properties involving
this sequence, including the Binet-style formula and the generating function are also
presented.
\end{abstract}

\medskip
\noindent
\subjclass{\footnotesize {\bf Mathematical subject classification 2010:} 
11B37, 11B39, 11B83.}

\medskip
\noindent
\keywords{\footnotesize {\bf Keywords:} Recurrence relation, Modified third-order Jacobsthal numbers, Third-order Jacobsthal numbers.}
\medskip

\vspace{5pt}

\section{Introduction}\label{sec:1}
\setcounter{equation}{0}
The Jacobsthal numbers have many interesting properties and applications in many fields of science (see, e.g., \cite{Ba,Hor2,Hor3}). The Jacobsthal numbers $J_{n}$ are defined by the recurrence relation
\begin{equation}\label{e1}
J_{0}=0,\ J_{1}=1,\ J_{n+2}=J_{n+1}+2J_{n},\ n\geq0.
\end{equation}
Another important sequence is the Jacobsthal-Lucas sequence. This sequence is defined by the recurrence relation $j_{n+2}=j_{n+1}+2j_{n}$, where $j_{0}=2$ and $j_{1}=1$ (see, \cite{Hor3}).

In \cite{Cook-Bac} the Jacobsthal recurrence relation is extended to higher order recurrence relations and the basic list of identities provided by A. F. Horadam \cite{Hor3} is expanded and extended to several identities for some of the higher order cases. For example, the third-order Jacobsthal numbers, $\{J_{n}^{(3)}\}_{n\geq0}$, and third-order Jacobsthal-Lucas numbers, $\{j_{n}^{(3)}\}_{n\geq0}$, are defined by
\begin{equation}\label{e2}
J_{n+3}^{(3)}=J_{n+2}^{(3)}+J_{n+1}^{(3)}+2J_{n}^{(3)},\ J_{0}^{(3)}=0,\ J_{1}^{(3)}=J_{2}^{(3)}=1,\ n\geq0,
\end{equation}
and 
\begin{equation}\label{e3}
j_{n+3}^{(3)}=j_{n+2}^{(3)}+j_{n+1}^{(3)}+2j_{n}^{(3)},\ j_{0}^{(3)}=2,\ j_{1}^{(3)}=1,\ j_{2}^{(3)}=5,\ n\geq0,
\end{equation}
respectively.

Some of the following properties given for third-order Jacobsthal numbers and third-order Jacobsthal-Lucas numbers are used in this paper (for more details, see \cite{Cer,Cer1,Cer2,Cook-Bac}). Note that Eqs. (\ref{e7}) and (\ref{e12}) have been corrected in this paper, since they have been wrongly described in \cite{Cook-Bac}.
\begin{equation}\label{e4}
3J_{n}^{(3)}+j_{n}^{(3)}=2^{n+1},
\end{equation}
\begin{equation}\label{e5}
j_{n}^{(3)}-3J_{n}^{(3)}=2j_{n-3}^{(3)},\ n\geq3,
\end{equation}
\begin{equation}\label{ec5}
J_{n+2}^{(3)}-4J_{n}^{(3)}=\left\{ 
\begin{array}{ccc}
-2 & \textrm{if} & \mymod{n}{1}{3} \\ 
1 & \textrm{if} & \mynotmod{n}{1}{3}
\end{array}%
\right. ,
\end{equation}
\begin{equation}\label{e6}
j_{n}^{(3)}-4J_{n}^{(3)}=\left\{ 
\begin{array}{ccc}
2 & \textrm{if} & \mymod{n}{0}{3} \\ 
-3 & \textrm{if} & \mymod{n}{1}{3}\\ 
1 & \textrm{if} & \mymod{n}{2}{3}%
\end{array}%
\right. ,
\end{equation}
\begin{equation}\label{e7}
j_{n+1}^{(3)}+j_{n}^{(3)}=3J_{n+2}^{(3)},
\end{equation}
\begin{equation}\label{e8}
j_{n}^{(3)}-J_{n+2}^{(3)}=\left\{ 
\begin{array}{ccc}
1 & \textrm{if} & \mymod{n}{0}{3} \\ 
-1 & \textrm{if} & \mymod{n}{1}{3} \\ 
0 & \textrm{if} & \mymod{n}{2}{3}%
\end{array}%
\right. ,
\end{equation}
\begin{equation}\label{e9}
\left( j_{n-3}^{(3)}\right) ^{2}+3J_{n}^{(3)}j_{n}^{(3)}=4^{n},
\end{equation}
\begin{equation}\label{e10}
\sum\limits_{k=0}^{n}J_{k}^{(3)}=\left\{ 
\begin{array}{ccc}
J_{n+1}^{(3)} & \textrm{if} & \mynotmod{n}{0}{3} \\ 
J_{n+1}^{(3)}-1 & \textrm{if} & \mymod{n}{0}{3}%
\end{array}%
\right. 
\end{equation}
and
\begin{equation}\label{e12}
\left( j_{n}^{(3)}\right) ^{2}-9\left( J_{n}^{(3)}\right)^{2}=2^{n+2}j_{n-3}^{(3)},\ n\geq3.
\end{equation}

Using standard techniques for solving recurrence relations, the auxiliary equation, and its roots are given by 
$$x^{3}-x^{2}-x-2=0;\ x = 2,\ \textrm{and}\ x=\frac{-1\pm i\sqrt{3}}{2}.$$ 

Note that the latter two are the complex conjugate cube roots of unity. Call them $\omega_{1}$ and $\omega_{2}$, respectively. Thus the Binet formulas can be written as
\begin{equation}\label{b1}
J_{n}^{(3)}=\frac{2}{7}2^{n}-\frac{3+2i\sqrt{3}}{21}\omega_{1}^{n}-\frac{3-2i\sqrt{3}}{21}\omega_{2}^{n}
\end{equation}
and
\begin{equation}\label{b2}
j_{n}^{(3)}=\frac{8}{7}2^{n}+\frac{3+2i\sqrt{3}}{7}\omega_{1}^{n}+\frac{3-2i\sqrt{3}}{7}\omega_{2}^{n},
\end{equation}
respectively. Now, we use the notation
\begin{equation}\label{h1}
V_{n}^{(2)}=\frac{A\omega_{1}^{n}-B\omega_{2}^{n}}{\omega_{1}-\omega_{2}}=\left\{ 
\begin{array}{ccc}
2 & \textrm{if} & \mymod{n}{0}{3} \\ 
-3 & \textrm{if} & \mymod{n}{1}{3} \\ 
1& \textrm{if} & \mymod{n}{2}{3}
\end{array}%
\right. ,
\end{equation}
where $A=-3-2\omega_{2}$ and $B=-3-2\omega_{1}$. Furthermore, note that for all $n\geq0$ we have 
\begin{equation}\label{im}
V_{n+2}^{(2)}=-V_{n+1}^{(2)}-V_{n}^{(2)},\ V_{0}^{(2)}=2\ \textrm{and}\ V_{1}^{(2)}=-3.
\end{equation}

From the Binet formulas (\ref{b1}), (\ref{b2}) and Eq. (\ref{h1}), we have
\begin{equation}\label{h2}
J_{n}^{(3)}=\frac{1}{7}\left(2^{n+1}-V_{n}^{(2)}\right)\ \textrm{and}\ j_{n}^{(3)}=\frac{1}{7}\left(2^{n+3}+3V_{n}^{(2)}\right).
\end{equation}

Motivated essentially by the recent works \cite{Cook-Bac}, \cite{Cer} and \cite{Cer2}, in this paper we introduce the Modified third-order Jacobsthal sequences and we give some properties, including the Binet-style formula and the generating functions for these sequences. Some identities involving these sequences are also provided.

\section{The Modified Third-order Jacobsthal sequence, Binet's formula and the generating function}
\setcounter{equation}{0}
The principal goals of this section will be to define the Modified third-order Jacobsthal sequence and to present some elementary results involving it.

First of all, we define the Modified third-order Jacobsthal sequence, denoted by $\{K_{n}^{(3)}\}_{n\geq0}$, which first terms are $\{3,1,3,10,15,31,66,...\}$. This sequence is defined recursively by
\begin{equation}\label{mean}
K_{n+3}^{(3)}=K_{n+2}^{(3)}+K_{n+1}^{(3)}+2K_{n}^{(3)},
\end{equation}
with initial conditions $K_{0}^{(3)}=3$, $K_{1}^{(3)}=1$ and $K_{2}^{(3)}=3$. Note that $K_{n}^{(3)}=J_{n}^{(3)}+2J_{n-1}^{(3)}+6J_{n-2}^{(3)}$, where $J_{n}^{(3)}$ is the $n$-th third-order Jacobsthal number.

In order to find the generating function for the Modified third-order Jacobsthal sequence, we shall write the sequence as a power series where each term of the sequence correspond to coefficients of the series. As a consequence of the definition of generating function of a sequence, the generating
function associated to $\{K_{n}^{(3)}\}_{n\geq0}$, denoted by $\{g_{K_{n}^{(3)}}(t)\}$, is defined by $$g_{K_{n}^{(3)}}(t)=\sum_{n\geq 0}K_{n}^{(3)}t^{n}.$$ Consequently, we obtain the following result:

\begin{theorem}
The generating function for the Modified third-order Jacobsthal numbers $\{K_{n}^{(3)}\}_{n\geq0}$
is $g_{K_{n}^{(3)}}(t)=\frac{3-2t-t^{2}}{1-t-t^{2}-2t^{3}}$.
\end{theorem}
\begin{proof}
Using the definition of generating function, we have $g_{K_{n}^{(3)}}(t)=K_{0}^{(3)}+K_{1}^{(3)}t+K_{2}^{(3)}t^{2}+\cdots+K_{n}^{(3)}t^{n}+\cdots$. Multiplying both sides of this identity by $-t$, $-t^{2}$ and by $-2t^{3}$, and then from (\ref{mean}), we have $(1-t-t^{2}-2t^{3})g_{K_{n}^{(3)}}(t)=K_{0}^{(3)}+(K_{1}^{(3)}-K_{0}^{(3)})t+(K_{2}^{(3)}-K_{1}^{(3)}-K_{0}^{(3)})t^{2}$ and the result follows.
\end{proof}

The following result gives the Binet-style formula for $K_{n}^{(3)}$.
\begin{theorem}\label{binetmod}
For $n\geq 0$, we have $K_{n}^{(3)}=2^{n}+\omega_{1}^{n}+\omega_{2}^{n}=2^{n}+M_{n}^{(2)}$, where 
\begin{equation}\label{mod}
M_{n}^{(2)}=\left\{ 
\begin{array}{ccc}
2 & \textrm{if} & \mymod{n}{0}{3} \\ 
-1 & \textrm{if} & \mymod{n}{1,2}{3}
\end{array}
\right.
\end{equation}
and $\omega_{1}, \omega_{2}$ are the roots of the characteristic equation associated with the respective recurrence relations $x^{2}+x+1=0$.
\end{theorem}
\begin{proof}
Since the characteristic equation has three distinct roots, the sequence $K_{n}^{(3)}=a2^{n}+b\omega_{1}^{n}+c\omega_{2}^{n}$ is the solution of the Eq. (\ref{mean}). Considering $n=0,1,2$ in this identity and solving this system of linear equations, we obtain a unique value for $a$, $b$ and $c$, which are, in this case, $a=b=c=1$. So, using these values in the expression of $K_{n}^{(3)}$ stated before, we get the required result.
\end{proof}

Using the fact that $\omega_{1}+\omega_{2}=-\omega_{1}\omega_{2}=-1$, we have 
\begin{equation}\label{mod1}
M_{n}^{(2)}=-\frac{1}{7}\left(4V_{n+1}^{(2)}-V_{n}^{(2)}\right)
\end{equation}
and $V_{n}^{(2)}$ as in Eq. (\ref{h1}). Then, $M_{n}^{(2)}=-M_{n-1}^{(2)}-M_{n-2}^{(2)}$, $M_{0}^{(2)}=2$ and $M_{1}^{(2)}=-1$. Furthermore,we easily obtain the identities stated in the following result:

\begin{proposition}\label{prop}
For a natural number $n$ and $m$, if $K_{n}^{(3)}$, $j_{n}^{(3)}$, and $K_{n}^{(3)}$ are, respectively,
the $n$-th third-order Jacobsthal, third-order Jacobsthal-Lucas and Modified third-order Jacobsthal numbers, then the following identities are true:
\begin{equation}\label{pp1}
147J_{n}^{(3)}=13K_{n}^{(3)}+48K_{n-1}^{(3)}+20K_{n-2}^{(3)},
\end{equation}
\begin{equation}\label{pp2}
6K_{n}^{(3)}=5j_{n}^{(3)}+3j_{n-1}^{(3)}-5j_{n-2}^{(3)},
\end{equation}
\begin{equation}\label{pp3}
49j_{n}^{(3)}=43K_{n}^{(3)}+8K_{n-1}^{(3)}+36K_{n-2}^{(3)},
\end{equation}
\begin{equation}\label{pp4}
K_{n}^{(3)}K_{m}^{(3)}+K_{n+1}^{(3)}K_{m+1}^{(3)}+K_{n+2}^{(3)}K_{m+2}^{(3)}=\left\lbrace \begin{array}{c}
21\cdot 2^{n+m}\\
+2^{n}\left(M_{m+1}^{(2)}+3M_{m+2}^{(2)}\right)\\
+2^{m}\left(M_{n+1}^{(2)}+3M_{n+2}^{(2)}\right)\\
+3(\omega_{1}^{n}\omega_{2}^{m}+\omega_{1}^{m}\omega_{2}^{n})
\end{array}
\right\rbrace ,
\end{equation}
\begin{equation}\label{pp5}
\left(K_{n}^{(3)}\right)^{2}+\left(K_{n+1}^{(3)}\right)^{2}+\left(K_{n+2}^{(3)}\right)^{2}=
21\cdot 2^{2n}+2^{n+1}\left(M_{n+1}^{(2)}+3M_{n+2}^{(2)}\right)+6,
\end{equation}
and $M_{n}^{(2)}$ as in Eq. (\ref{mod1}).
\end{proposition}
\begin{proof}
First, we will just prove Eqs. (\ref{pp1}) and (\ref{pp4}) since Eqs. (\ref{pp2}), (\ref{pp3}) and (\ref{pp5}) can be dealt with in the same manner. 

(\ref{pp1}): To prove Eq. (\ref{pp1}), we use induction on $n$. Let $n=2$, we get $$147J_{2}^{(3)}=147=13\cdot 3+48\cdot 1+20\cdot 3=13K_{2}^{(3)}+48K_{1}^{(3)}+20K_{0}^{(3)}.$$ Let us assume that $147J_{m}^{(3)}=13K_{m}^{(3)}+48K_{m-1}^{(3)}+20K_{m-2}^{(3)}$ is true for all values $m$ less than or equal $n$. Then, 
\begin{align*}
147J_{n+1}^{(3)}&=147\left(J_{n}^{(3)}+J_{n-1}^{(3)}+2J_{n-2}^{(3)}\right)\\
&=13K_{n}^{(3)}+48K_{n-1}^{(3)}+20K_{n-2}^{(3)}\\
&\ \ +13K_{n-1}^{(3)}+48K_{n-2}^{(3)}+20K_{n-3}^{(3)}\\
&\ \ +26K_{n-2}^{(3)}+96K_{n-3}^{(3)}+40K_{n-4}^{(3)}\\
&=13K_{n+1}^{(3)}+48K_{n}^{(3)}+20K_{n-1}^{(3)}.
\end{align*}

(\ref{pp4}): Using the the Binet formula of $K_{n}^{(3)}$ in Theorem \ref{binetmod}, we have
\begin{align*}
K_{n}^{(3)}K_{m}^{(3)}&+K_{n+1}^{(3)}K_{m+1}^{(3)}+K_{n+2}^{(3)}K_{m+2}^{(3)}\\
&=\left\lbrace \begin{array}{c}
\left(2^{n}+M_{n}^{(2)}\right)\left(2^{m}+M_{m}^{(2)}\right)\\
+\left(2^{n+1}+M_{n+1}^{(2)}\right)\left(2^{m+1}+M_{m+1}^{(2)}\right)\\
+\left(2^{n+2}+M_{n+2}^{(2)}\right)\left(2^{m+2}+M_{m+2}^{(2)}\right)
\end{array}
\right\rbrace .
\end{align*}
Then,
\begin{align*}
K_{n}^{(3)}K_{m}^{(3)}&+K_{n+1}^{(3)}K_{m+1}^{(3)}+K_{n+2}^{(3)}K_{m+2}^{(3)}\\
&=\left\lbrace \begin{array}{c}
21\cdot 2^{n+m}+2^{n}\left(M_{m}^{(2)}+2M_{m+1}^{(2)}+4M_{m+2}^{(2)}\right)\\
+2^{m}\left(M_{n}^{(2)}+2M_{n+1}^{(2)}+4M_{n+2}^{(2)}\right)\\
+M_{n}^{(2)}M_{m}^{(2)}+M_{n+1}^{(2)}M_{m+1}^{(2)}+M_{n+2}^{(2)}M_{m+2}^{(2)}
\end{array}
\right\rbrace\\
&=\left\lbrace \begin{array}{c}
21\cdot 2^{n+m}+2^{n}\left(M_{m+1}^{(2)}+3M_{m+2}^{(2)}\right)\\
+2^{m}\left(M_{n+1}^{(2)}+3M_{n+2}^{(2)}\right)\\
+3(\omega_{1}^{n}\omega_{2}^{m}+\omega_{1}^{m}\omega_{2}^{n})
\end{array}
\right\rbrace .
\end{align*}
Then, we obtain the Eq. (\ref{pp5}) if $m=n$ in Eq. (\ref{pp4}). 
\end{proof}

\section{Some identities involving this type of sequence}
\setcounter{equation}{0}
In this section, we state some identities related with these type of third-order sequence. As a consequence of the Binet formula of Theorem \ref{binetmod}, we get for this sequence the following interesting identities.

\begin{proposition}[Catalan's identities]\label{prop2}
For a natural numbers $n$, $s$, with $n\geq s$, if $K_{n}^{(3)}$ is the $n$-th Modified third-order Jacobsthal numbers, then the following identity $$K_{n+s}^{(3)}K_{n-s}^{(3)}-\left(K_{n}^{(3)}\right)^{2}=\left\lbrace \begin{array}{c}
2^{n}\left(2^{-s}-2^{s}\right)U_{s}^{(2)}M_{n+1}^{(2)}\\
- 2^{n}\left(2^{s}U_{s}^{(2)}+2+\left(2^{-s}+2^{s}\right)U_{s-1}^{(2)}\right)M_{n}^{(2)}\\
-3 \left(U_{s}^{(2)}\right)^{2}
\end{array}
\right\rbrace$$ is true, where $M_{n}^{(2)}$ as in Eq. (\ref{mod1}), $U_{n}^{(2)}=\frac{\omega_{1}^{n}-\omega_{2}^{n}}{\omega_{1}-\omega_{2}}$ and $\omega_{1}$, $\omega_{2}$ are the roots of the characteristic equation associated with the recurrence relation $x^{2}+x+1=0$.
\end{proposition}
\begin{proof}
Using the Eq. (\ref{mod1}) of Proposition \ref{prop} and the Binet formula of $K_{n}^{(3)}$ in Theorem \ref{binetmod}, we have
\begin{align*}
K_{n+s}^{(3)}K_{n-s}^{(3)}-\left(K_{n}^{(3)}\right)^{2}&=\left\lbrace \begin{array}{c}
\left(2^{n+s}+M_{n+s}^{(2)}\right)\left(2^{n-s}+M_{n-s}^{(2)}\right)\\
+\left(2^{n}+M_{n}^{(2)}\right)^{2}
\end{array}
\right\rbrace\\
&=\left\lbrace \begin{array}{c}
2^{n}\left(2^{s}M_{n-s}^{(2)}-2M_{n}^{(2)}+2^{-s}M_{n+s}^{(2)}\right)\\
+M_{n+s}^{(2)}M_{n-s}^{(2)}-\left(M_{n}^{(2)}\right)^{2}
\end{array}
\right\rbrace.
\end{align*}
Using that $M_{n+s}^{(2)}=U_{s}^{(2)}M_{n+1}^{(2)}-U_{s-1}^{(2)}M_{n}^{(2)}$, $U_{s}^{(2)}=\frac{\omega_{1}^{s}-\omega_{2}^{s}}{\omega_{1}-\omega_{2}}$ and $U_{-s}^{(2)}=-U_{s}^{(2)}$. Then, we obtain
$$
K_{n+s}^{(3)}K_{n-s}^{(3)}-\left(K_{n}^{(3)}\right)^{2}=\left\lbrace \begin{array}{c}
2^{n}\left(2^{-s}-2^{s}\right)U_{s}^{(2)}M_{n+1}^{(2)}\\
- 2^{n}\left(2^{s}U_{s}^{(2)}+2+\left(2^{-s}+2^{s}\right)U_{s-1}^{(2)}\right)M_{n}^{(2)}\\
-3 \left(U_{s}^{(2)}\right)^{2}
\end{array}
\right\rbrace .$$ Hence the result.
\end{proof}

Note that for $s=1$ in Catalan's identity obtained, we get the Cassini identity for the Modified third-order Jacobsthal sequence. In fact, for $s=1$, the identity stated in Proposition \ref{prop2}, yields $$ K_{n+1}^{(3)}K_{n-1}^{(3)}-\left(K_{n}^{(3)}\right)^{2}=\left\lbrace \begin{array}{c}
2^{n}\left(2^{-1}-2^{1}\right)U_{1}^{(2)}M_{n+1}^{(2)}\\
- 2^{n}\left(2^{1}U_{1}^{(2)}+2+\left(2^{-1}+2^{1}\right)U_{0}^{(2)}\right)M_{n}^{(2)}\\
-3 \left(U_{1}^{(2)}\right)^{2}
\end{array}
\right\rbrace $$ and using one of the initial conditions of the sequence $\{U_{n}^{(2)}\}$ in Proposition \ref{prop2} we obtain the following result.

\begin{proposition}[Cassini's identities]
For a natural numbers $n$, if $K_{n}^{(3)}$ is the $n$-th Modified third-order Jacobsthal numbers, then the identity $$K_{n+1}^{(3)}K_{n-1}^{(3)}-\left(K_{n}^{(3)}\right)^{2}=2^{n-1}\left(3M_{n+2}^{(2)}-5M_{n}^{(2)}\right)-3$$ is true.
\end{proposition}

The d'Ocagne identity can also be obtained using the Binet formula and in this case we obtain
\begin{proposition}[d'Ocagne's identities]
For a natural numbers $m$, $n$, with $m\geq n$ and $K_{n}^{(3)}$ is the $n$-th Modified third-order Jacobsthal number, then the following identity
$$K_{m+1}^{(3)}K_{n}^{(3)}-K_{m}^{(3)}K_{n+1}^{(3)}=\left\lbrace \begin{array}{c}
2^{m}\left(2M_{n}^{(2)}-M_{n+1}^{(2)}\right)\\
+ 2^{n}\left(M_{m+1}^{(2)}-2M_{m}^{(2)}\right)-3 U_{m-n}^{(2)}
\end{array}
\right\rbrace$$ is true.
\end{proposition}
\begin{proof}
Using the Eq. (\ref{mod1}) of Proposition \ref{prop} and the Eq. (\ref{binetmod}) of Theorem \ref{binetmod}, we get the required result.
\end{proof}

In addition, some formulae involving sums of terms of the Modified third-order Jacobsthal sequence will be provided in the following proposition.

\begin{proposition}
For a natural numbers $m$, $n$, with $n\geq m$, if $j_{n}^{(3)}$ and $K_{n}^{(3)}$ are, respectively, the $n$-th third-order Jacobsthal-Lucas and Modified third-order Jacobsthal numbers, then the following identities are true:
\begin{equation}\label{t1}
\sum_{s=m}^{n}K_{s}^{(3)}=\frac{1}{3}\left(K_{n+2}^{(3)}+2K_{n}^{(3)}+K_{m}^{(3)}-K_{m+2}^{(3)}\right),
\end{equation}
\begin{equation}\label{t2}
\sum_{s=0}^{n}K_{s}^{(3)}=\left\{ 
\begin{array}{ccc}
K_{n+1}^{(3)}+2 & \textrm{if} & \mymod{n}{0}{3} \\ 
K_{n+1}^{(3)}+1 & \textrm{if} & \mymod{n}{1}{3}\\
K_{n+1}^{(3)}-3 & \textrm{if} & \mymod{n}{2}{3}
\end{array}
\right. ,
\end{equation}
\begin{equation}\label{t3}
\sum_{s=0}^{n}j_{s}^{(3)}=\frac{1}{49}\left(16K_{n+3}^{(3)}-5K_{n+2}^{(3)}+2K_{n+1}^{(3)}\right)-1.
\end{equation}
\end{proposition}
\begin{proof}
(\ref{t1}): Using Eq. (\ref{mean}), we obtain
\begin{align*}
\sum_{s=m}^{n}K_{s}^{(3)}&=K_{m}^{(3)}+K_{m+1}^{(3)}+K_{m+2}^{(3)}+\sum_{s=m+3}^{n}K_{s}^{(3)}\\
&=K_{m}^{(3)}+K_{m+1}^{(3)}+K_{m+2}^{(3)}\\
&\ \ +\sum_{s=m+2}^{n-1}K_{s}^{(3)}+\sum_{s=m+1}^{n-2}K_{s}^{(3)}+2\sum_{s=m}^{n-3}K_{s-3}^{(3)}\\
&=4\sum_{s=m}^{n}K_{s}^{(3)}+K_{m+2}^{(3)}-K_{m}^{(3)}-4K_{n}^{(3)}-3K_{n-1}^{(3)}-2K_{n-2}^{(3)}\\
&=4\sum_{s=m}^{n}K_{s}^{(3)}+K_{m+2}^{(3)}-K_{m}^{(3)}-2K_{n}^{(3)}-K_{n+2}.
\end{align*}
Then, the result in Eq. (\ref{t1}) is completed.

(\ref{t2}): As a consequence of the Eq. (\ref{mod}) of Theorem \ref{binetmod} and $$\sum_{s=0}^{n}M_{s}^{(2)}=\frac{1}{3}\left(M_{n}^{(2)}-M_{n+1}^{(2)}\right)+1,$$ we have 
\begin{align*}
\sum_{s=0}^{n}K_{s}^{(3)}&=\sum_{s=0}^{n}2^{s}+\sum_{s=0}^{n}M_{s}^{(2)}\\
&=2^{n+1}-1+\frac{1}{3}\left(M_{n}^{(2)}-M_{n+1}^{(2)}\right)+1\\
&=K_{n+1}^{(3)}+\frac{1}{3}\left(M_{n}^{(2)}-4M_{n+1}^{(2)}\right)\\
&=\left\{ 
\begin{array}{ccc}
K_{n+1}^{(3)}+2 & \textrm{if} & \mymod{n}{0}{3} \\ 
K_{n+1}^{(3)}+1 & \textrm{if} & \mymod{n}{1}{3}\\
K_{n+1}^{(3)}-3 & \textrm{if} & \mymod{n}{2}{3}
\end{array}
\right. .
\end{align*}
Hence we obtain the result.

(\ref{t3}): Using Eqs. (\ref{pp3}) and (\ref{mean}), we have $49j_{n}^{(3)}=43K_{n}^{(3)}+8K_{n-1}^{(3)}+36K_{n-2}^{(3)}=18K_{n+1}^{(3)}+25K_{n}^{(3)}-10K_{n-1}^{(3)}$. Then, from Eq. (\ref{t1}), we obtain
\begin{align*}
49\sum_{s=0}^{n}j_{s}^{(3)}&=98+18\sum_{s=1}^{n}K_{s+1}^{(3)}+25\sum_{s=1}^{n}K_{s}^{(3)}-10\sum_{s=1}^{n}K_{s-1}^{(3)}\\
&=98+18\sum_{s=2}^{n+1}K_{s}^{(3)}+25\sum_{s=1}^{n}K_{s}^{(3)}-10\sum_{s=0}^{n-1}K_{s}^{(3)}\\
&=98+6\left(K_{n+3}^{(3)}+2K_{n+1}^{(3)}+K_{2}^{(3)}-K_{4}^{(3)}\right)\\
&\ \ +\frac{25}{3}\left(K_{n+2}^{(3)}+2K_{n}^{(3)}+K_{1}^{(3)}-K_{3}^{(3)}\right)\\
&\ \ -\frac{10}{3}\left(K_{n+1}^{(3)}+2K_{n-1}^{(3)}+K_{0}^{(3)}-K_{2}^{(3)}\right)\\
&=16K_{n+3}^{(3)}-5K_{n+2}^{(3)}+2K_{n+1}^{(3)}-49.
\end{align*}
So, the proof is completed. 
\end{proof}

For negative subscripts terms of the sequence of Modified third-order Jacobsthal we can establish the following result:
\begin{proposition}
For a natural number $n$ the following identities are true:
\begin{equation}\label{n1}
K_{-n}^{(3)}=K_{n}^{(3)}+2^{-n}-2^{n},
\end{equation}
\begin{equation}\label{n2}
\sum_{s=0}^{n}K_{-s}^{(3)}=\frac{1}{3}\left(K_{n+2}^{(3)}+2K_{n}^{(3)}\right)-2^{n+1}-2^{-n}+3.
\end{equation}
\end{proposition}
\begin{proof}
(\ref{n1}): Since $M_{-n}^{(2)}=M_{n}^{(2)}$, using the Binet formula stated in Theorem \ref{binetmod} and the fact that $\omega_{1}\omega_{2}=1$, all the results of this Proposition follow. In fact,
\begin{align*}
K_{-n}^{(3)}&=2^{-n}+M_{-n}^{(2)}=2^{-n}+M_{n}^{(2)}\\
&=2^{-n}+2^{n}+M_{n}^{(2)}-2^{n}\\
&=K_{n}^{(3)}+2^{-n}-2^{n}.
\end{align*}
So, the proof is completed.

(\ref{n2}): The proof is similar to the proof of Eq. (\ref{t1}) using the Eq. (\ref{n1}).
\end{proof}

\section{Conclusion}
Sequences of numbers have been studied over several years, with emphasis on the well known Tribonacci sequence and, consequently, on the Tribonacci-Lucas sequence. In this paper we have also contributed for the study of Modified third-order Jacobsthal sequence, deducing some formulae for the sums of such numbers, presenting the generating functions and their Binet-style formula. It is our intention to continue the study of this type of sequences, exploring some their applications in the science domain. For example, a new type of sequences in the quaternion algebra with the use of this numbers and their combinatorial properties.

\medskip

\end{document}